\documentclass[10pt]{amsart}

\usepackage{amssymb,amsmath,amsthm,amsfonts}
\usepackage{pdfsync}

\newcommand{\comment}[1]{}

\newtheorem{lem}{Lemma}
\newtheorem{propn}{Proposition}

\newtheorem{thm}{Theorem}

\newtheorem*{thmA}{Theorem A}
\newtheorem*{thmB}{Corollary B}
\newtheorem*{thmC}{Corollary C}
\theoremstyle{remark}
\newtheorem*{Rem}{Remark}
\theoremstyle{definition}

\newcommand{\R}{\mathbb R}

\newcommand{\Z}{\mathbb Z}
\newcommand{\T}{\mathbb T}

\newcommand{\N}{\mathbb N}
\newcommand{\C}{\mathbb C}

\DeclareMathOperator{\supp}{supp}

\DeclareMathOperator{\lcm}{lcm}

\newcommand{\vp}{\varphi}
\newcommand{\D}{\delta}

\newcommand{\VE}{\varepsilon}
\newcommand{\A}{\alpha}
\newcommand{\B}{\beta}
\newcommand{\lm}{\lambda}

\newcommand{\be}{\begin{equation}}
\newcommand{\ee}{\end{equation}}
\newcommand{\bee}{\begin{equation*}}
\newcommand{\eee}{\end{equation*}}

\begin{document}
\title{Optimal Polynomial Recurrence}
\author{Neil Lyall\quad\quad\quad\'Akos Magyar}
\thanks{Both authors were partially supported by NSF grants.}

\address{Department of Mathematics, The University of Georgia, Athens, GA 30602, USA}
\email{lyall@math.uga.edu}
\address{Department of Mathematics, University of British Columbia, Vancouver, B.C. V6T 1Z2, Canada}
\email{magyar@math.ubc.ca}


\begin{abstract}
Let $P\in\Z[n]$ with $P(0)=0$ and $\VE>0$.
We show, using Fourier analytic techniques, that if $N\geq \exp\exp(C\VE^{-1}\log\VE^{-1})$ and $A\subseteq\{1,\dots,N\}$, then there must exist $n\in\N$ such that
\[\frac{|A\cap (A+P(n))|}{N}>\left(\frac{|A|}{N}\right)^2-\VE.\]

In addition to this we also show, using the same Fourier analytic methods, that if $A\subseteq\N$, then the set of \emph{$\VE$-optimal return times}
\[R(A,P,\VE)=\left\{n\in \N \,:\,\D(A\cap(A+P(n)))>\D(A)^2-\VE\right\}\] 
is syndetic for every $\VE>0$. Moreover, we show that $R(A,P,\VE)$ is \emph{dense} in every sufficiently long interval, in particular we show that
 there exists an $L=L(\VE,P,A)$ such that 
\[\left|R(A,P,\VE)\cap I\right|
\geq c(\VE,P)|I|\] 
for all intervals $I$ of natural numbers with $|I|\geq L$ with $c(\VE,P)=\exp\exp(-C\,\VE^{-1}\log\VE^{-1})$.
\end{abstract}
\maketitle

\setlength{\parskip}{4pt}


\section{Introduction.}
\subsection{Background}
The study of recurrence properties of dynamical systems goes back to the beginnings of ergodic theory. 
If $A$ is a measurable subset of a probability space $(X,\mathcal{M},\mu)$ with $\mu(A)>0$ and $T$ is a measure preserving transformation, then it was already shown by Poincar\'e \cite{P} that $\mu(A\cap T^{-n}A)>0$ for some  natural number $n$, and hence for infinitely many.

Poincar\'e's result was subsequently sharpened by Khintchine \cite{K}, who showed that sets of positive measure not only return to intersect themselves infinitely often, but in fact return ``frequently'' with ``large'' intersection.  In order to be more precisely we
recall that a set $R\subseteq\N$ is said to be \emph{syndetic} if it has bounded gaps: there
exists $L\in\N$ such that every interval of length greater than $L$ intersects $R$ non-trivially. A precise formulation of Khintchine's result is that for every $\VE>0$, the set
\be
\left\{n\in \N \,:\,\mu(A\cap T^{-n})>\mu(A)^2-\VE\right\}
\ee
is syndetic. 
Note that in general this lower bound is sharp, since $\mu(A\cap T^{-n} A)\to\mu(A)^2$ as $n\to \infty$ whenever $T$ is a mixing transformation.

The following polynomial version of Khintchine's result, where the set of natural numbers $n$ is replaced by the values of a polynomial $P\in\Z[n]$ satisfying $P(0)=0$, was established by Furstenberg \cite{F}.

\begin{thmA}[Furstenberg \cite{F}] 
Let $(X,\mathcal{M},\mu,T)$ be an invertible measure preserving system, $A\in\mathcal{M}$ and $P\in\Z[n]$ with $P(0)=0$.
For every $\VE>0$, the set
\be
\{n\in\N\,:\,\mu(A\cap T^{-P(n)} A)>\mu(A)^2-\VE\}
\ee
is syndetic.
\end{thmA}
For a proof of this result we refer the reader to the presentations in \cite{M} or \cite{B}, see also \cite{LM3}. It follows from (two different variants of) Furstenberg's correspondence principle that Theorem A has the following two combinatorial consequences.

\begin{thmB}[Furstenberg \cite{F}, see Frantzikinakis and Kra \cite{FK}]
Let $A\subseteq\N$ and $P\in\Z[n]$ with $P(0)=0$. For every $\VE>0$, the set
\be
\left\{n\in \N \,:\,\D(A\cap(A+P(n)))>\D(A)^2-\VE\right\}
\ee
is syndetic, where $\D(B)=\limsup_{N\rightarrow\infty}|B\cap[1,N]|/N$ denotes the upper density of a given set $B\subseteq\N$.
\end{thmB}

\begin{thmC}[Furstenberg \cite{F}, see Frantzikinakis and Kra \cite{FK}]
Let $P\in\Z[n]$ with $P(0)=0$. For every $\VE>0$ there exists $N_1=N_1(\VE,P)$ such that if $N\geq N_1$ and $A\subseteq [1,N]$, then there exists $n\in\N$ such that
\be
\frac{|A \cap(A+P(n))|}{N} >\left(\frac{|A|}{N}\right)^2-\VE.\ee
\end{thmC}

We note that this correspondence gives no quantitative bounds in the finite setting of Corollary C (other than the special case when the polynomial is linear).

\begin{Rem} Recently, far reaching generalizations of Furstenberg's results (Theorem A and its corollaries) have been obtained in the settings of multiple recurrence. In particular, if $(X,\mathcal{M},\mu,T)$ is an invertible measure preserving system, $A\in\mathcal{M}$ and $P_1,\ldots,P_\ell$ be any linearly independent family of integral polynomials with $P_i(0)=0$ for all $1\leq i\leq \ell$, then Frantzikinakis and Kra \cite{FK} have shown that for every $\VE>0$, the set
\[\{n\in\N\,:\,\mu(A\cap T^{-P_1(n)}A \cap\cdots\cap T^{-P_\ell(n)}A)>\mu(A)^{\ell+1}-\VE\}\]
is syndetic, and that the analogous extensions of Corollaries B and C also hold.

A study of the (intermediate) phenomenon of simultaneous (single) polynomial recurrence was initiated by the authors in \cite{LM3}, see also \cite{LM1} and \cite{LM1'}.

For a comprehensive survey of the impact of the Poincar\'e recurrence principle in ergodic theory, especially as pertains to the field of ergodic Ramsey theory/additive combinatorics, see \cite{FM}, \cite{B2} and \cite{Kra}.
\end{Rem}

\subsection{Statement of Main Results}

The purpose of this article is to establish, using Fourier analytic methods, the following quantitative versions of Corollaries C and B.

\begin{thm}\label{finite1}
Let $A\subseteq[1,N]$, $P\in\Z[n]$ with $P(0)=0$ and $\VE>0$, then
\bee
\left|\left\{n\in[0,L]\,:\, \frac{|A\cap(A+P(n))|}{N}> \left(\dfrac{|A|}{N}\right)^2-\VE\right\}\right|\geq c(\VE,P)L
\eee
for all $1\leq L\leq N^{1/k}$ where $k=\deg(P)$ and $c(\VE,P)=\exp\exp(-C\VE^{-1}\log\VE^{-1})$.
\end{thm}
Note that in order to obtain a non-trivial conclusion from Theorem \ref{finite1} we must have $L\geq c(\VE,P)^{-1}$ and consequently also $N\geq c(\VE,P)^{-k}$. In particular, this implies Corollary C with $N_1(\VE, P)=\exp\exp(C\VE^{-1}\log\VE^{-1})$.
\begin{Rem}
In a supplementary document \cite{LM2'} we give a proof of Corollary C with the quantitative bound $N_1(\VE, n^2)=\exp\exp(C\VE^{-1}\log\VE^{-1})$, in the special case where the polynomial $P(n)=n^2$.  While the presentation in \cite{LM2'} closely follows that of the current paper, many of the analogous arguments are significantly less technical and as such we feel that the reader may find the exposition in \cite{LM2'} illuminating.
 \end{Rem}

\begin{thm}\label{infinite1}
Let $A\subseteq\N$, $P\in\Z[n]$ with $P(0)=0$ and $\VE>0$, then 
there exists $L=L(\VE,P,A)$ such that 
\bee
\left|\left\{n\in I \,:\,\D(A\cap(A+P(n)))>\D(A)^2-\VE\right\}\right|
\geq c(\VE,P)|I|
\eee
for all intervals $I$ of natural numbers with $|I|\geq L$ and $c(\VE,P)=\exp\exp(-C\,\VE^{-1}\log\VE^{-1})$.
\end{thm}

We note that the parameter $L$ in Theorem \ref{infinite1} necessarily depends on the actual set $A$ in question and not just on its density, for a proof of this fact see Section \ref{A}.

We remark that Theorem \ref{infinite1} (and Corollary B) also holds if one replaces the upper density $\D$ with the upper Banach density $\D^*$ defined for $A\subseteq\N$ by $\D^*(A)=\lim_{N\rightarrow\infty}\sup_{x\in\N}|A\cap(x+[1,N])|/N$.

The strategy we will employ to prove Theorems \ref{finite1} and \ref{infinite1} is to \emph{lift} the problem in such a way that we may then apply the following analogous higher dimensional results.

\begin{thm}\label{finite2}
Let $B\subseteq[1,M]^k$, $\gamma(n)=(n,n^2,\dots,n^k)$ and $\VE>0$, then  
\bee
\left|\left\{n\in[0,K]\,:\, \frac{|B\cap(B+\gamma(n))|}{M^k}> \left(\dfrac{|B|}{M^k}\right)^2-\VE\right\}\right|\geq c(\VE,k)K
\eee
for all $1\leq K\leq M^{1/k}$ with $c(\VE,k)=\exp\exp(-C\VE^{-1}\log\VE^{-1})$.
\end{thm}

\begin{thm}\label{infinite2}
Let $B\subseteq\N^k$, $\gamma(n)=(n,n^2,\dots,n^k)$ and $\VE>0$, then there exists $K=K(\VE,k,B)$ such that 
\bee
\left|\left\{n\in I \,:\,\D(B\cap(B+\gamma(n)))>\D(B)^2-\VE\right\}\right|
\geq c(\VE,k)|I|
\eee
for all intervals $I$ of natural numbers with $|I|\geq K$ and $c(\VE,k)=\exp\exp(-C\VE^{-1}\log\VE^{-1})$.
\end{thm}

Recall that for $B\subseteq\N^k$ the upper density of $B$ is defined to be $\D(B)=\limsup_{M\rightarrow\infty}|B\cap[1,M]^k|/M^k.$

\subsection{An outline of the paper}\label{outline}
The bulk of the present paper is concerned with establishing Theorems \ref{finite2} and \ref{infinite2}, from which Theorems \ref{finite1} and \ref{infinite1} follow in an essentially straightforward manner. These deductions are presented in Sections \ref{finite} and \ref{infinite} respectively.

In Section \ref{pre} we recall some basic properties of the Fourier transform on $\Z^k$, see how these can be used to count differences in $B\subseteq[1,M]^k$ of the form $\gamma(t)$ and lead us naturally to the analysis of certain variants of standard Weyl sums. 

In Section \ref{red} we reduce the task of proving Theorems \ref{finite2} and \ref{infinite2} to a key dichotomy proposition, namely Proposition \ref{dichotomy2}. The proofs that Proposition \ref{dichotomy2} implies Theorems \ref{finite2} and \ref{infinite2} are presented in Sections \ref{prooffinite} and \ref{proofinfinite} respectively. The arguments in these sections are close in spirit, and very much influenced by, those of Bourgain \cite{Bourgain}, see also Magyar \cite{Magyar}. 

In order to use Fourier analysis to prove Proposition \ref{dichotomy2} we introduce a smooth functional variant of  Proposition \ref{dichotomy2}, namely Proposition \ref{smoothdichotomy}. The proof of Proposition \ref{smoothdichotomy} as well as the proof that it implies Proposition \ref{dichotomy2} is given in Sections \ref{4} and \ref{5}.

Finally, in Section \ref{A}, we include a short proof of the fact that the parameter $L$ in Theorem \ref{infinite1} necessarily depends on the actual set $A$ in question and not just its density.


\section{Preliminaries}\label{pre}

\subsection{Fourier analysis on $\Z^k$}
If $f:\Z^k\rightarrow\C$ is a function for which $\sum_{m\in\Z^k}|f(m)|<\infty$
we will say that $f\in L^1(\Z^k)$ and define 
\[\|f\|_1=\sum_{m\in\Z^k}|f(m)|.\]
For $f\in L^1$ we define its \emph{Fourier transform} $\widehat{f}:\T^k\rightarrow\C$ by
\[\widehat{f}(\A)= \sum\limits_{m\in\Z^k}f(m)e^{-2\pi i m\cdot\A}.\]
Note that the summability assumption on $f$ ensures that the series defining $\widehat{f}$ converges uniformly to a continuous function on $\T^k$ (which we will identify with the unit cube $[0,1]^k$ in $\R^k$) and that the Fourier inversion formula and Plancherel's identity, namely
\[f(m)=\int_{\T^k}\widehat{f}(\A)e^{2\pi i m\cdot\A}d\A\quad\quad\text{and}\quad\quad\int_{\T^k}|\widehat{f}(\A)|^2 d\A=\sum_{m\in\Z^k}|f(m)|^2\]
are, in this setting, immediate consequences of the familiar orthogonality relation
\[\int_{\T^k}e^{2\pi i m\cdot\A}d\A=\begin{cases}
1\quad\text{if \ $m=0$}\\ 
0\quad\text{if \ $m\ne0$}
\end{cases}.\]

Defining the convolution of $f$ and $g$ to be
\[f*g(m)=\sum_{\ell\in\Z^k}f(m-\ell)g(\ell)\]
it follows that if $f,g\in L^1$ then $f*g\in L^1$ with
\[\|f*g\|_1\leq\|f\|_1\|g\|_1\quad\quad\text{and}\quad\quad\widehat{f*g}=\widehat{f}\,\widehat{g}.\]

Finally, we remark that it follows from the Poisson Summation Formula that if $\vp\in\mathcal{S}(\R^k)$, then
\be
\widehat{\vp}(\A)=\sum_{\ell\in\Z^k}\widetilde{\vp}(\A-\ell)
\ee
where
\be\label{FTonRk}
\widetilde{\vp}(\xi)=\int_{\R^k}\vp(x)e^{-2\pi i x\cdot\xi}\,dx
\ee
denotes the Fourier transform (on $\R^k$) of $\vp$.

\subsection{Counting differences of the form $\gamma(n)$
}

Let $B\subseteq[1,M]^k$ and $\D=|B|/M^k$. 

Let $1\leq\mu\leq\lm$ be integers with $\lm^k\leq M/2^k$.
It is easy to verify, using the properties of the Fourier transform discussed above, that the average number of pairs of elements in $B$ whose difference is equal to $\gamma(n)$ with
$n\in(\lm,\lm+\mu]\cap\Z$ can be expressed as follows:
\be\frac{1}{\mu}\sum_{n=\lm+1}^{\lm+\mu}|B\cap(B+\gamma(n))|=\frac{1}{\mu}\sum_{n=\lm+1}^{\lm+\mu}\sum_{m\in\Z^k}1_B(m)1_B(m-\gamma(n))=\int_{\T^k}|\widehat{1_B}(\A)|^2 S_{\lm,\mu}(\A)\,d\A\ee
where
\be\label{weyl sum 2}
S_{\lm,\mu}(\A)=\frac{1}{\mu}\sum_{n=\lm+1}^{\lm+\mu} e^{2\pi i \A\cdot\gamma(n)}.
\ee
It is easy to see that
\be\label{weyl relate 1}
S_{\lm,\mu}(\A)=\frac{\lm+\mu}{\mu}\,S_{\lm+\mu}(\A)-\frac{\lm}{\mu}\,S_\lm(\A)
\ee
where
\be\label{weyl sum}
S_\mu(\A)=\frac{1}{\mu}\sum_{n=1}^{\mu} e^{2\pi i \A\cdot\gamma(n)}
\ee
denotes a classical (normalized) Weyl sum.
Unfortunately, the rather simplistic relationship indicated in (\ref{weyl relate 1}) will only be useful to us in the case where $\mu=\lm$. When $\mu<\lm$ we will make use of the following alternative:
\be\label{weyl relate 2}
S_{\lm,\mu}(\A)=e^{2\pi i \A\cdot\gamma(\lm)}S_\mu(T_\lm\A)
\ee
where $T_\lm$ is a $k\times k$ matrix whose entries are given by
 \be
(T_{\lm })_{ij}=\begin{cases}
\binom{j}{i}\lm^{j-i}&\quad j\geq i\\
0 &\quad j<i
\end{cases}.
\ee

\subsection{Standard Weyl sum estimates}

It is clear that whenever $|\A_j|\ll \mu^{-j}$ there can be no cancellation in the Weyl sum (\ref{weyl sum}), in fact it is easy to verify that the same is also true whenever each $\A_j$ is \emph{close} to a rational with \emph{small} denominator (there is no cancellation over sums in residue classes modulo $q$). 

We now state a precise formulation of the well known fact that this is indeed the only obstruction to cancellation. 
For a proof of this result see either \cite{LM1} or \cite{LM1'}.

\begin{lem}\label{Weyl Estimates}
Let $\eta>0$ and  $\mu\geq \eta^{-C}$ (with $C$ sufficiently large depending on $k$). If for some $1\leq j\leq k$ we have
\be
\left|\A_j-\frac{a}{q}\right|>\frac{1}{\eta^k\mu^j}
\ee 
for all $a\in\Z$ and $1\leq q\leq \eta^{-k}$,
then
\be\label{minorest}
\left|S_{\mu}(\A)\right|\leq C_1\eta.
\ee
\end{lem}

\begin{Rem}
It is easy to see that one can conclude from Lemma \ref{Weyl Estimates} that estimate (\ref{minorest}) also holds (under the same hypotheses as above with $C_1$ replaced with say $2C_1$) for the ``perturbed'' Weyl sums
\[\frac{1}{\mu}\sum_{n\in(\lm,\lm+\mu]\cap\Z} e^{2\pi i \A\cdot \gamma(n)}\]
where $1\leq\mu\leq\lm$ are now no longer assumed to take on integer values, provided $\mu\gg\eta^{-1}$.
\end{Rem}

We note that from Lemma \ref{Weyl Estimates}, relationship (\ref{weyl relate 2}) and the Plancherel identity, we may conclude that 
\[\int_{\T^k}|\widehat{1_B}(\A)|^2 S_{\lm,\mu}(\A)\,d\A=\int_{T_\lm^{-1}\mathfrak{M}_{\eta,\mu}} |\widehat{1_B}(\A)|^2 S_{\lm,\mu}(\A)\,d\A+O(\eta M^k)\]
where \[\mathfrak{M}_{\eta,\mu}=\bigcup_{q=1}^{\eta^{-k}}\left\{\A\in\T^k\,:\, \left|\A_j-\frac{a_j}{q}\right|\leq \frac{1}{\eta^k\mu^j}  \ \text{($1\leq j\leq k$) for some $a\in\mathbb{Z}^k$}\right\}.\]
While in the case $\mu=\lm$ 
it follows from (\ref{weyl relate 1}) that 
\[|S_{\lm,\lm}(\A)|\leq 3C_1\eta\]
whenever $\A\notin\mathfrak{M}_{\eta,\lm}$ and as a consequence of this
we can in fact make the rather more favorable conclusion that
\[\int_{\T^k}|\widehat{1_B}(\A)|^2 S_{\lm,\lm}(\A)\,d\A=\int_{\mathfrak{M}_{\eta,\lm}} |\widehat{1_B}(\A)|^2 S_{\lm,\lm}(\A)\,d\A+O(\eta M^k).\]


In order to carry out our Fourier analytic arguments it will be convenient to consider the (nonisotropic) lattice 
\[\left\{\left(\frac{a_1}{q_\eta},\frac{a_2}{q^2_\eta},\dots,\frac{a_k}{q_\eta^k}\right)\in\T^k\,:\, (a_1,\dots,a_k)\in\Z^k\right\}\]
of rational points where 
\be
q_\eta=\lcm\{1\leq q \leq \eta^{-k}\}
\ee 
as opposed to the much smaller, but alas more wildly distributed, set of rational points that appear as the centers of the major boxes in $\T^k$ that constitute $\mathfrak{M}_{\eta,\mu}$.  
Note that it follows from elementary considerations involving the prime numbers that 
\be 
q_\eta\leq \exp(C\eta^{-k})
\ee
and this accounts for one of the exponentials in the bound in Theorems \ref{finite2} and \ref{infinite2} (as well as \ref{finite1} and \ref{infinite1}).

\section{Reduction to dichotomy propositions}\label{red}

We now separately present the statement of two key propositions (although as we shall see the first of which follows immediately from the second)
and demonstrate how they can be used to prove Theorem \ref{finite2} and Theorem \ref{infinite2} respectively.

For $L>1$ and $q\in\N$ we define 
\be\label{box}
M_{q,L}=\left\{\A\in\T^k\,:\, \left|\A_j-\frac{a_j}{q^j}\right|\leq \frac{1}{L^j}  \ \text{($1\leq j\leq k$) for some $a\in\mathbb{Z}^k$}\right\}.
\ee
Let $\eta>0$ and $1\leq\mu\leq\lm$. We define
\be
\Omega_{\eta,\lm,\mu}=\left\{\A\in\T^k\,:\, \A\in M_{q_\eta,\eta^k \mu}\setminus M_{q_\eta,\eta^{-k}\lm}\right\}
\ee
where $q_\eta=\lcm\{1\leq q \leq \eta^{-k}\}$ as before.

\subsection{Proof of Theorem \ref{finite2}}\label{prooffinite}
Although this result can in fact be deduced from the second dichotomy proposition (Proposition \ref{dichotomy2} below), we feel that the reduction of Theorem \ref{finite2} to the (simpler) Proposition \ref{dichotomy1} is not only more direct and straightforward (by virtue of the fact that we can take $\mu=\lm$), but that our decision to include it will also serve to illuminate the deduction of Theorem \ref{infinite2} from Proposition \ref{dichotomy2}.

\begin{propn}\label{dichotomy1} Let $B\subseteq[1,M]^k$ and $\VE>0$.
Let $\eta_\VE=\exp(-C\VE^{-1}\log\VE^{-1})$ and $q_\VE=q_{\eta_\VE}$.

If $\lm$ is an integer that satisfies
$\lm\geq \eta_\VE^{-k}q_\VE$ and $M\geq C(\eta_\VE^{-k}\lm)^k$
then either
\be\label{random1}
\left|\left\{n\in(\lm,2\lm]\cap\Z\,:\,\frac{|B\cap(B+\gamma(n))|}{M^k}> \left(\dfrac{|B|}{M^k}\right)^2-\VE \right\}\right|\geq \exp(-C\eta_\VE^{-k})\lm
\ee
or
\be\label{roughint1}
\int_{\Omega} |\widehat{1_B}(\A)|^2\,d\A\geq\VE M^k/10
\ee
where $\Omega=\Omega_{\eta_\VE,\lm,\lm}$.
\end{propn}

Proposition \ref{dichotomy1} (and Proposition \ref{dichotomy2} below) both express, in our setting, the basic dichotomy that either $B$ behaves as though it were a random set, or has arithmetic structure as the Fourier transform $\widehat{1_B}$ is concentrated (on small annuli) around a fixed (nonisotropic) lattice of rational points.

\begin{proof}[Proof of Theorem \ref{finite2}]
Let $\VE>0$, $\eta_\VE=\exp(-C\VE^{-1}\log\VE^{-1})$ and $q_\VE=q_{\eta_\VE}$. 
Suppose $K$ and $M$ are integers that satisfy \[\exp(C\eta_\VE^{-k})\leq K\leq M^{1/k}\]
and $\{\lm_j\}_{j=1}^J$ is a sequence of integers with $J>10/\VE$ with the property that $\lm_1\geq \eta_\VE^{-k} q_\VE$, $\lm_J=c\eta_\VE^{k}K$ and 
\be\label{disjoint1}
\eta_\VE^{-2k}\lm_j\leq  \lm_{j+1}\leq C\eta_\VE^{-2k}\lm_j
\ee
for $1\leq j\leq J$. 
It is easy to now see that the sets
$\Omega_j=\Omega_{\eta_\VE,\lm_j,\lm_j}$
are disjoint.

Suppose, contrary to Theorem \ref{finite2}, that there does exists a set $B\subseteq[1,M]^k$ such that
\be\label{negative}
\left|\left\{n\in[0,K]\,:\, \frac{|B\cap(B+\gamma(n))|}{M^k}> \left(\dfrac{|B|}{M^k}\right)^2-\VE\right\}\right|<\exp(-C\eta_\VE^{-k})K\ee
for all $C>0$. Since $K/\lm_j\leq C\eta_\VE^{-2kJ}\ll \exp(-C\eta_\VE^{-k})$ for all $1\leq j\leq J$, it follows that
\be\label{dyadic}
\left|\left\{n\in(\lm_j,2\lm_j]\cap\Z\,:\, \frac{|B\cap(B+\gamma(n))|}{M^k}> \left(\dfrac{|B|}{M^k}\right)^2-\VE\right\}\right|<\exp(-C\eta_\VE^{-k})\lm_j\ee
for all $C>0$ and all $1\leq j\leq J$. 

Proposition \ref{dichotomy1} allows us to conclude from this that
\be
\sum_{j=1}^J\int_{\Omega_j} |\widehat{1_B}(\A)|^2\,d\A\geq J\VE M^k/10>M^k.
\ee

On the other hand it follows from the disjointness property of the sets $\Omega_j$ (which we guaranteed by our initial choice of sequence $\{\lm_j\}$) and the Plancherel identity that
\be
\sum_{j=1}^J\int_{\Omega_j} |\widehat{1_B}(\A)|^2\,d\A\leq \int_{\T^k}|\widehat{1_B}(\A)|^2\,d\A\leq |B|\leq M^k
\ee
giving us our desired contradiction. 
\end{proof}

\subsection{Proof of Theorem \ref{infinite2}}\label{proofinfinite}
 We now present the statement of our second (stronger) dichotomy proposition.

\begin{propn}\label{dichotomy2} Let $B\subseteq[1,M]^k$ and $\VE>0$.
Let $\eta_\VE=\exp(-C\VE^{-1}\log\VE^{-1})$ and $q_\VE=q_{\eta_\VE}$.

If $1\leq\mu\leq\lm$ are any given pair of integers that satisfy
$\mu\geq \eta_\VE^{-k}q_\VE$ and $M\geq C(\eta_\VE^{-k}\lm)^k$
then either
\be\label{random2}
\left|\left\{n\in(\lm,\lm+\mu]\cap\Z\,:\,\frac{|B\cap(B+\gamma(n))|}{M^k}> \left(\dfrac{|B|}{M^k}\right)^2-\VE \right\}\right|\geq \exp(-C\eta_\VE^{-k})\mu
\ee
or
\be\label{roughint2}
\int_{T_{\lm}^{-1}\Omega} |\widehat{1_B}(\A)|^2\,d\A\geq\VE M^k/10
\ee
where $\Omega=\Omega_{\eta_\VE,\lm,\mu}$.
\end{propn}

Key to deducing Theorem \ref{infinite2}
from Proposition \ref{dichotomy2} is the following combinatorial result on the annuli $\Omega_{\eta_\VE,\lm,\mu}$.

\begin{lem}[Overlapping Lemma]\label{cor} 
Let $\eta>0$. Suppose $\{\mu_j\}_{j\in\N}$ and $\{\lm_j\}_{j\in\N}$ be sequences such that $\mu_1\geq \eta^{-k}q_\eta$ and
\be\label{lac}
\mu_j \leq \lm_j\leq \frac{1}{3}\,\eta^{2k}\mu_{j+1}
\ee
for all $j\in\N$, then it follows that
\[\A\in T_{\lm_j}^{-1}\Omega_j\]
for at most $k$ different values of $j$, where $\Omega_j =\Omega_{\eta,\lm_j,\mu_j}$.
\end{lem}
The proof of this result is given in the subsection immediately preceding the proof of Theorem \ref{infinite2}.

\begin{proof}[Proof of Theorem \ref{infinite2}]
Let $\VE>0$, $\eta_\VE=\exp(-C\VE^{-1})$ and $q_\VE=q_{\eta_\VE}$. Suppose, contrary to Theorem \ref{infinite2}, that there exists a set $B\subseteq\N^k$ with $\D=\D(B)>\VE^{1/2}$ such that for all sufficiently long intervals of natural numbers $I$ one has
\be\label{negative2}
\left|\left\{n\in I \,:\,\D(B\cap(B+\gamma(n)))>\D(B)^2-\VE\right\}\right|
<\exp(-C\eta_\VE^{-k})|I|
\ee
for all $C>0$. 

In this case there necessarily exists a sequence of intervals of natural numbers $I_j=(\lm_j,\lm_j+\mu_j]$ with $\mu_1\geq4q_\eta$ and $\mu_j\nearrow\infty$ for which
\be\label{anyinterval}
\left|\left\{n\in I_j \,:\,\D(B\cap(B+\gamma(n)))>\D(B)^2-\VE\right\}\right|
<\exp(-C\eta_\VE^{-k})|I_j|
\ee   
for all $C>0$ and all $j\in\N$. 

Since inequality (\ref{anyinterval}) must then also hold for the \emph{right-half} intervals $I_j'=(\lm_j',\lm_j'+\mu_j']$, where $\mu_j'=\mu_j/2$ and $\lm_j'=\lm_j+\mu_j'$, we see that we can further assume that $\lm_j\rightarrow\infty$.  By passing to a subsequence, one 
may without loss in generality assume that $\mu_1\geq \eta^{-k}q_\VE$ and
\bee
\mu_j \leq \lm_j\leq \frac{1}{3}\,\eta^{2k}\mu_{j+1}
\eee
for all $j\in\N$.

We now fix an integer $J>40k/\VE$.
It follows from the definition of upper density that there must exist $M\in\N$ such that
\[\left|B\cap[1,M]^k\right|\geq (\D-\VE/2) M^k\]
while
\[\left|B\cap(B+\gamma(n))\cap[1,M]^k\right|\leq(\D^2-\VE)M^k\]      
for all $n\in\bigcup_{j=1}^J  I_j$ for which $\D(B\cap(B+\gamma(n)))\leq\D(B)^2-\VE$.

Letting $B'=B\cap[1,M]^k$ it follows that
\[\left|\left\{n\in I_j\,:\, \frac{\left|B'\cap(B'+\gamma(n))\right|}{M^k}>\left(\frac{B'}{M^k}\right)^2-\frac{\VE}{4}\right\}\right|<\exp(-C\eta_\VE^{-k})|I_j|\]
for all $C>0$ and all $1\leq j\leq J$. 

Proposition \ref{dichotomy2} allows us to conclude from this that
\be
\sum_{j=1}^J\int_{T_{\lm_j}^{-1}\Omega_j} |\widehat{1_B}(\A)|^2\,d\A\geq J\VE M^k/40>kM^k.
\ee

On the other hand it follows from Lemma \ref{cor} (with $\Omega_j =\Omega_{\eta_\VE,\lm_j,\mu_j}$) and the Plancherel identity that
\be
\sum_{j=1}^J\int_{T_{\lm_j}^{-1}\Omega_j} |\widehat{1_B}(\A)|^2\,d\A\leq k\int_{\T^k}|\widehat{1_B}(\A)|^2\,d\A\leq k|B|\leq kM^k
\ee
giving us our desired contradiction. 
\end{proof}

\subsection{Proof of the Overlapping Lemma}\label{OL} 

 First we establish the following.
\begin{lem}\label{L3} Suppose that $0<\eta<1/4k^{2}$ and $0<\mu\leq\lm$. If $\A\in T_{\lm}^{-1}\Omega$, where  $\Omega=\Omega_{\eta,\lm,\mu}$, then there exist $1\leq i\leq k$ and $a\in\Z$ such that
\be
\frac{1}{2}\left(\frac{\eta^k}{\lm}\right)^{i}\leq\left|\A_{i}-\frac{a}{q_\eta^i}\right|\leq \frac{3}{2}\left(\frac{1}{\eta^k\mu}\right)^{i}.
\ee
\end{lem}

\begin{proof}[Proof of Lemma \ref{L3}]
Suppose that $T_\lm\A\in\Omega$, then for some $1\leq j\leq k$ we have 
\be\label{inside}
\left|(T_{\lm}\A)_j-\frac{a_j}{q_\eta^j}\right|\geq \left(\frac{\eta^k}{\lm}\right)^j
\ee
for all $a\in\Z^k$, while for all $1\leq j\leq k$ we have
\be\label{outside}
\left|(T_{\lm}\A)_j-\frac{a_j'}{q_\eta^j}\right|\leq\left(\frac{1}{\eta^k\mu}\right)^j
\ee
for some $a'\in\Z^k$.

Denote by $i$ is the largest integer from $\{1,\dots,k\}$ for which
\[\left|(T_{\lm}\A)_i-\frac{a}{q_\eta^{i-1}}\right|\geq \left(\frac{\eta^k}{\lm}\right)^i\]
for all $a\in\Z$. 

A key observation is that for each $1\leq j\leq k$ one can write
\[(T_\lm\A)_j=\A_j+ \binom{j+1}{j}\lm^{}\A_{j+1}+\cdots+\binom{k}{j}\lm^{k-j}\A_{k}=\A_j+c_j\lm (T_\lm\A)_{j+1}\] 
where $0<c_j\leq j(k-j)<k^2$, that the maximal assumption on $i$ ensures that 
\[\left|(T_{\lm}\A)_{i+1}-\frac{a_i''}{q_\eta^{i}}\right|< \left(\frac{\eta^k}{\lm}\right)^{i+1}\]
for some $a''\in\Z^k$, and hence that 
\be\label{key overlap}
\left|(T_{\lm}\A-\A)_{i}-\frac{a_i'''}{q_\eta^{i}}\right|< 2c_j\eta^k \left(\frac{\eta^k}{\lm}\right)^{i}\leq\frac{1}{2}\left(\frac{\eta^k}{\lm}\right)^{i}<\frac{1}{2}\left(\frac{1}{\eta^k\mu}\right)^{i}
\ee
for some $a'''\in\Z^k$.

We note that it follows immediately from (\ref{key overlap}) and (\ref{outside}) that
\[\left|\A_{i}-\frac{a_i}{q_\eta^i}\right|\leq \frac{3}{2}\left(\frac{1}{\eta^k\mu}\right)^{i}.\]
for some $a\in\Z^k$, while from (\ref{key overlap}) and (\ref{inside}) it follows that for all $a\in\Z^k$ we have
\[\left|\A_{i}-\frac{a_i}{q_\eta^i}\right|\geq \frac{1}{2}\left(\frac{\eta^k}{\lm}\right)^{i}.\qedhere\]
\end{proof}

\begin{proof}[Proof of Lemma \ref{cor}]

If $T_{\lm_j}\A\in \Omega_j$, then Lemma \ref{L3} guarantees the existence of an integer $1\leq i_j\leq k$ such that
\be\label{trapped}
\frac{1}{2}\left(\frac{\eta^k}{\lm_j}\right)^{i_j}\leq\left|\A_{i_j}-\frac{a}{q}\right|\leq \frac{3}{2}\left(\frac{1}{\eta^k\mu_j}\right)^{i_j}
\ee
for some $a\in\Z$.

Suppose there exists $\A\in\T^k$ and distinct integers $j_1,\dots,j_{k+1}$ for which
\[\A\in T_{\lm_{j_1}}^{-1}\Omega_{j_1}\cap\dots\cap T_{\lm_{j_{k+1}}}^{-1}\Omega_{j_{k+1}}.\]
It follows from the pigeonhole principle that there must exists integers $j,j'\in\{j_1,\dots,j_{k+1}\}$, with $j<j'$, for which $i_j=i_{j'}$. Inequality (\ref{trapped}) and the fact that $\mu_1\geq \eta^{-k}q_\eta$ the forces the situation that
\[\eta^{2k}\mu_{j'}< 3\lm_{j}\]
which contradicts (\ref{lac}).
\end{proof}


\section{Formulation of smooth variants of Propositions \ref{dichotomy1} and \ref{dichotomy2}}\label{4}

We now formulate \emph{smooth} functional variants of Proposition \ref{dichotomy1} and \ref{dichotomy2} that are better suited to our Fourier analytic approach.

\subsection{Counting function}


For $g,h:[1,M]^k\rightarrow[0,1]$ and $q,\lm,\mu\in\N$ we define

\be
\Lambda_{q,\mu}(g,h)=\frac{q}{\mu}\sum_{\substack{n\in(\lm,\lm+\mu] \\ q|n}}\sum_{m\in\Z^k}g(m)h(m-\gamma(n)).
\ee

With $g=h=1_B$ this essentially gives a normalized count for the number of pairs of elements in $B$ whose difference is equal to $\gamma(n)$ with
$n\in(\lm,\lm+\mu]\cap\Z$ and $q|n$.

Note that it is natural to consider only those $n\in\N$ that are divisible by some (large) natural number $q$. Indeed, as a consequence of the fact that our set $B$ could fall entirely into a subset of $\Z^k$ of the form $x+d\,\Z\times\Z^{k-1}$ with $1\leq d\le\VE^{-1/2}$, it follows that if there were to exist $n\in\N$ such that $B\cap(B+\gamma(n))\ne\emptyset$ for an arbitrary set $B$, then these $n$ would necessarily have to be divisible by all $1\leq d\leq \VE^{-1/2}$ and hence by the least common multiple of all $1\leq d\leq \VE^{-1/2}$, a quantity of size $\exp(C\VE^{-1/2})$.

As before this can be expressed this count on the transform side as
\be\label{FTside}
\Lambda_{q,\mu}(g,h)=\int_{\T^k} \widehat{g}(\A)\overline{\widehat{h}(\A)}S_{\lm,\mu,q}(\A)\,d\A
\ee
where 
\be
S_{\lm,\mu,q}(\A)=\frac{q}{\mu}\sum_{\substack{n\in(\lm,\lm+\mu] \\ q|n}}e^{2\pi i\,\A\cdot\gamma(n)}.
\ee

\begin{Rem}
If the integers $\lm$ and $\mu$ are both divisible by $q$, then one can easily verify that
\be\label{relate3}
S_{\lm,\mu,q}(\A)=S_{\lm/q,\mu/q}(q\circ\A)
\ee
where 
\be
q\circ\A=(q\A_1,\dots,q^k\A_k)
\ee 
and as such we can deduce estimates for these new exponential sums, via relations (\ref{weyl relate 1}) and (\ref{weyl relate 2}), from those that are stated in Lemma \ref{Weyl Estimates}. See in particular Lemma \ref{Weyl Estimates 2} below.
\end{Rem}

\subsection{Smooth variants of our dichotomy propositions 
}

Let $\vp:\R^k\rightarrow(0,\infty)$ be a Schwartz function satisfying 
\[\widetilde{\vp}(0)=1\geq\widetilde{\vp}(\xi)\geq0\quad\quad\text{and}\quad\quad \widetilde{\vp}(\xi)=0 \ \ \text{for} \ \ |\xi|>1\]
where $\widetilde{\vp}$ denotes the Fourier transform (on $\R^k$) of $\vp$, see (\ref{FTonRk}).

For a given $q\in\N$ and $L>1$ we define
\be\label{normalizedpsi}
\vp_{q,L}(x)=\begin{cases}
\left(\frac{q}{L}\right)^{k(k-1)/2}\vp\left(\frac{q\ell_1}{L},\dots,\frac{q^k\ell_k}{L^k}\right)& \ \ \text{if} \ \ x=(q\ell_1,\dots,q^k\ell_k) \ \ \text{for some $\ell\in\Z^k$}\\
 \ 0& \ \ \text{otherwise}
\end{cases}
\ee
It follows from the Poisson summation formula that the Fourier transform (on $\Z^k$) of $\vp_{q,L}$ takes the form
\be
\widehat{\vp}_{q,L}(\A)
=\sum_{\ell\in\Z^k}\widetilde{\vp}\left(L\left(\A_1-\frac{\ell_1}{q}\right),\dots,L^k\left(\A_k-\frac{\ell_k}{q^k}\right)\right).
\ee

We now define 
\[\psi_{q,L}(x)=\vp_{q,L}(T_{\lm}^{*^{-1}}x)\]
where $T_{\lm}^*$ denotes the adjoint of $T_\lm$,
and note that
\[\widehat{\psi}_{q,L}(\A)=\widehat{\vp}_{q,L}(T_{\lm} \A).\]

Note that $\widehat{\vp}_{q,L}$ is supported on $M_{q,L}$ (and hence $\widehat{\psi}_{q,L}$ is supported on $T_\lm^{-1} M_{q,L}$), where $M_{q,L}$ are the major boxes defined by (\ref{box}),
and that we may choose our cutoff function $\vp$ so that
\be
\widehat{\vp}_{q_\VE,\eta_\VE^k \mu}-\widehat{\vp}_{q_\VE,\VE\eta_\VE^{-k}\lm}
\ee
will be essentially supported on $\Omega_{\eta_\VE,\lm,\mu}$
in the sense that 
\be\label{cutoff}
\bigl| \widehat{\vp}_{q_\VE,\eta_\VE^k \mu}(\A)-\widehat{\vp}_{q_\VE,\VE\eta_\VE^{-k}\lm}(\A)\bigr|\leq\VE/10
\ee
whenever $\A\notin\Omega_{\eta_\VE,\lm,\mu}$.


The smooth variant of Proposition  \ref{dichotomy2} is then the following:

\begin{propn}[Smooth variant of Proposition \ref{dichotomy2}]\label{smoothdichotomy}
Let $f:[1,M]^k\rightarrow[0,1]$ and set $\D=M^{-k}\sum_{m\in\Z^k}f(m)$.

Let $0<\VE\leq\D^2$ and $1\leq\mu\leq\lm$ be any given pair of integers that satisfy
$\mu\geq\eta_\VE^{-k}q_\VE$ and $M\geq C(\eta_\VE^{-k}\lm)^k$
where
$q_\VE=q_{\eta_\VE}$ with $\eta_\VE=\exp(-C\VE^{-1}\log\VE^{-1})$. 
Then there exists $0<\eta\ll\VE$ satisfying
$\eta_\VE\leq\VE\eta$, such that either
\be\label{SR}
\Lambda_{q,\mu}(f,f)>(\D^2-\VE)M^k
\ee
or
\be\label{smoothint}
\int_{\T^k}|\widehat{f}(\A)|^2\bigl|\widehat{\psi}_{q,L_2}(\A)-\widehat{\psi}_{q,L_1}(\A)\bigr|\,d\A\geq\VE M^k/5
\ee
where $L_1=\eta^{-k}\lm$, $L_2=\eta^k \mu$, and $q=q_\eta$.
\end{propn}

\begin{Rem} We have chosen to not explicitly state the analogous smooth variant of Proposition \ref{dichotomy1}, since this would be simply Proposition \ref{smoothdichotomy} with $\mu$ set equal to $\lm$ and $\psi$ replaced with $\vp$. 
\end{Rem}


We finish this section by explicitly showing that Proposition \ref{smoothdichotomy} does indeed imply Proposition \ref{dichotomy2}, the same argument of course also establishes that Proposition \ref{dichotomy1} would follow from its (unstated) analogous smooth variant. 

\begin{proof}[Proof that Proposition \ref{smoothdichotomy} implies Proposition \ref{dichotomy2}]
Let $f=1_B$ and $q=q_\eta$, noting that and $q\leq q_\VE$.  

It is easy to see that if $\Lambda_{q,\mu}(f,f)>(\D^2-\VE)M^k$, then
\be
\left|\left\{n\in(\lm,\lm+\mu]\cap\Z \,:\,|B\cap(B+\gamma(n))|>(\D^2-2\VE)M^k\right\}\right|\geq \frac{c\VE}{q}\mu\geq \frac{c\VE}{q_\VE}\mu
\ee
which immediately gives (\ref{random2}), with $2\VE$ in place of $\VE$, since $q_\VE\leq \exp(C\eta_\VE^{-k})$.

While from the fact that $q|q_\VE$ it follows that
\[\supp\bigl(\widehat{\psi}_{q,L_2}-\widehat{\psi}_{q,L_1}\bigr)\subseteq \supp\bigl(\widehat{\psi}_{q_\VE,\eta_\VE^k \mu}-\widehat{\psi}_{q_\VE,\VE\eta_\VE^{-k}\lm}\bigr)\]
and hence from the remarks preceding Proposition \ref{smoothdichotomy} (in particular (\ref{cutoff})) that (\ref{smoothint}) implies (\ref{roughint2}).
\end{proof}

\section{Proof of Proposition \ref{smoothdichotomy}}\label{5}

We now present the proof of Proposition \ref{smoothdichotomy}, finally completing the proofs of Theorems \ref{finite2} and \ref{infinite2}. 
As opposed to the usual Fourier proofs of S\'ark\"ozy's theorem, which are based on density
increment arguments, here we use an energy increment argument, (in fact a regularity
lemma type decomposition) to obtain optimal recurrence.

\begin{Rem}
We have already noted that in order to establish Theorem \ref{finite2} we need only prove Proposition \ref{smoothdichotomy} with $\mu=\lm$ and $\psi$ replaced with $\vp$. Making these substitutions in the proof below will indeed give a proof of the (unstated) smooth variant of Proposition \ref{dichotomy1} (one must also, in the proof of Lemma \ref{Weyl Estimates 2}, (naturally) replace $T_\lm$ with the identity matrix and increase the size of some constants threefold).
\end{Rem}

\subsection{Decomposition}
Let
$f:[1,M]^k\rightarrow[0,1]$ and $\D=M^{-k}\sum_{m\in\Z^k}f(m)$.

We make the decomposition
\be
f=f_1+f_2+f_3
\ee
where
\be
f_1=f*\psi_{q,L_1}\quad\text{and}\quad
f_2=f-f*\psi_{q,L_2}
\ee
which of course forces 
\be
f_3=f*(\psi_{q,L_2}-\psi_{q,L_1}).
\ee

\comment{
We (formally) define
\be
\Lambda_q(g,h)=\frac{q}{\mu}\sum_{d\in(\lm,\lm+\mu]\cap q\Z}\sum_{m\in\Z^k}g(m)h(m-\gamma(d)).
\ee
}

One should think of $f_1(m)$ (respectively $f*\psi_{q,L_2}(m)$) as being essentially the average value of the function $f$ over arithmetic grids of the form $\{q_\eta\circ \ell\,: \ell\in\Z^k\}$ of (total) size $L_1\times L_1^2\times\cdots\times L_1^k$ (respectively $L_2\times L_2^2\times\cdots\times L_2^k$) centered at $m$.

\subsection{Proof of Proposition \ref{smoothdichotomy}}

Note that
\be
\Lambda_{q,\mu}(f,f)=\Lambda_{q,\mu}(f_1,f_1)+\underbrace{\Lambda_{q,\mu}(f_2,f_1)+\Lambda_{q,\mu}(f,f_2)}_{(\star)}+\underbrace{\Lambda_{q,\mu}(f_3,f_1)+\Lambda_{q,\mu}(f,f_3)}_{(\star\star)}
\ee
where both terms in $(\star)$ involve a $f_2$ and both terms in $(\star\star)$ involve a $f_3$.

The proof of Proposition \ref{smoothdichotomy} will follow as an almost immediate consequence of the following two lemmas.

\begin{lem}[Main term]\label{lemma1}
Let $\VE>0$. If $0<\eta\ll\VE$, then
\be
\Lambda_{q,\mu}(f_1,f_1)\geq (\D^2-\VE/2)M^k
\ee
\end{lem}

\begin{lem}[Error term]\label{lemma3}
Let $\VE>0$, then there exists $\eta>0$ satisfying $\exp(-C\VE^{-1}\log\VE^{-1})\leq\eta\ll\VE$, such that
\be\label{wholepoint}
\|(1-\widehat{\psi}_{q,L_2})S_{\lm,\mu,q}\|_\infty\leq\VE/20
\ee 
and hence
\be
|\Lambda_{q,\mu}(f_2,f_1)+\Lambda_{q,\mu}(f,f_2)|\leq (\VE/10)M^k.
\ee
\end{lem}

\begin{proof}[Proof of Proposition \ref{smoothdichotomy}]
If
$\Lambda_{q,\mu}(f,f)\leq(\D^2-\VE)M^k$, then it follows from
Lemma \ref{lemma1} that
\[|\Lambda_{q,\mu}(f,f)-\Lambda_{q,\mu}(f_1,f_1)|\geq(\VE/2)M^k.\]
Since
\[|\Lambda_{q,\mu}(f_3,f_1)+\Lambda_{q,\mu}(f,f_3)|\geq |\Lambda_{q,\mu}(f,f)-\Lambda_{q,\mu}(f_1,f_1)|-
|\Lambda_{q,\mu}(f_2,f_1)+\Lambda_{q,\mu}(f,f_2)|\]
it consequently follows from Lemma \ref{lemma3} that
\[|\Lambda_{q,\mu}(f_3,f_1)+\Lambda_{q,\mu}(f,f_3)|\geq (2\VE/5)M^k.\]

The proposition then follows from the observation that
\be\label{max}
\max\{|\Lambda_{q,\mu}(f_3,f_1)|,|\Lambda_{q,\mu}(f,f_3)|\}
\leq \int_{\T^k}|\widehat{f}(\A)|^2\bigl|\widehat{\psi}_{q,L_2}(\A)-\widehat{\psi}_{q,L_1}(\A)\bigr|\,d\A.
\ee
which follows from standard properties of convolutions under the action of the Fourier transform, identity (\ref{FTside}), and trivial bounds for the exponential sum $S_{\lm,\mu,q}$.\end{proof}

\subsection{Proof of Lemma \ref{lemma1}}

Let $q=q_\eta$.
If $q|n$ and $\lm<n\leq\lm+\mu$ (and hence $n\leq2\eta^k L_1$), then it is straightforward to see that $\vp$ can be chosen so that $f_1$ is essentially invariant under translation by $\gamma(n)$ in the the sense that
\be
\left|f_1(m)-f_1(m-\gamma(n))\right|=\left(\frac{q}{L}\right)^{k(k-1)/2}\sum_{\ell\in\Z^k}\left|\vp\left(\frac{q\ell_1-n}{L},\dots,\frac{q^k\ell_k-n^k}{L^k}\right)-\vp\left(\frac{q\ell_1}{L},\dots,\frac{q^k\ell_k}{L^k}\right)\right|\leq c\eta^k
\ee
for some constant $c>0$.

Therefore, provided $\eta$ is chosen so that $c\eta^k\leq \VE/4$, we have
\begin{align}
\Lambda_{q,\mu}(f_1)
&\geq \sum_{m\in\Z^k} f_1(m)^2 -\frac{\VE}{4}\sum_{m\in\Z^k}f_1(m).
\end{align}
Since $\psi_{q,L_1}$ is $L^1$-normalized it follows that
\be
\sum_{m\in\Z^k}f_1(m)=\sum_{m,\ell\in\Z^k}f(m-\ell)\psi_{q,L_1}(\ell)=\sum_{m\in\Z^k}f(m)=\D M^k.
\ee
Using Cauchy-Schwarz, one obtains
\be
\sum_{m\in\Z^k} f_1(m)^2\geq\sum_{-\sigma M\leq m_j\leq M+\sigma M} f_1(m)^2
\geq  \frac{1}{(1+2\sigma)^k M^k} \left(\sum_{-\sigma M\leq m_j\leq M+\sigma M} f_1(m)\right)^2
\ee 
Since $f$ is supported on $[1,M]^k$ (and $\psi_{q,L_1}$ is $L^1$-normalized) it follows that

\be
\sum_{-\sigma M\leq m_j\leq M+\sigma M} f_1(m)\geq \sum_{m\in\Z^k}f(m)\left(1-\sum_{|\ell_j|\geq\sigma M} \psi_{q,L_1}(\ell)\right)\geq \D M^k (1-\sigma)
\ee
as $\vp$ can be chosen so that $\sum_{|\ell_j|\geq\sigma M} \psi_{q,L_1}(\ell)\leq \sigma$ whenever $M\gg L_1$.

Note that \[\frac{(1-\sigma)^2}{(1+2\sigma)^k}\geq (1-\sigma)^2(1-2\sigma k)\geq1-4k\sigma\]
provided $2\sigma k<1$. Hence taking $\sigma=\VE/16k$ completes the proof.
\qed

\subsection{Proof of Lemma \ref{lemma3}}
It is in establishing Lemma \ref{lemma3} that we finally exploit the arithmetic properties of the curve $\gamma(n)$. In particular, we will make use of the following ``minor arc" estimates for the exponential sums $S_{\lm,\mu,q}$.

\begin{lem}[Corollary of Lemma \ref{Weyl Estimates}]\label{Weyl Estimates 2}
Let $\VE>0$. If $0<\eta\ll\VE$ and $0<\eta'<\VE\eta$, then
\be\label{MainSquares}
\|(1-\widehat{\psi}_{q',L'_2})S_{\lm,\mu,q}\|_\infty\leq 2C_1\eta'/\eta
\ee
where  $q'=q_{\eta'}$ and $L'_2=\eta'^k\mu$.
\end{lem}
\begin{proof}
Let $\eta_0=\eta'/\eta$ and  $\A\in\T^k$ be fixed. 
If there exists $a\in\Z^k$ such that
\[\left|(T_\lm\A)_i-\frac{a}{q'^i}\right|\leq\frac{\VE}{(\eta'^k\mu)^i},\]
for all $1\leq i\leq k$, then (as remarked earlier) $\vp$ can be chosen such that
\be\label{*}
|1-\widehat{\psi}_{q',L'_2}(\A)|\leq \VE.\ee

While if
for some $1\leq i\leq k$ we have
\bee
\left|(T_\lm\A)_i-\frac{a}{q'^i}\right|>\frac{\VE}{(\eta'^k\mu)^i}
\eee
for all $a\in\Z^k$,
then
\[\left|q^i(T_\lm\A)_i-\frac{a}{q_0^i}\right|>\frac{q^i}{(\eta_0^{k}\mu)^i}\]
for all $a\in\Z^k$, since $q q_0|q'$ where $q_0=q_{\eta_0}$.

Since 
\[q\circ(T_\lm\A)=T_{\lm/q}(q\circ\A)\]
it follows from (\ref{relate3}) that
\[S_{\lm,\mu,q}(\A)=\frac{1}{\mu'}\sum_{s\in(\lm'+q^{-1},\lm'+\mu']\cap\Z} e^{2\pi i \gamma(s)\cdot(q\circ\A)}\]
where $\lm'=\lm/q$ and $\mu'=\mu/q$ and the remark proceeding Lemma \ref{Weyl Estimates} 
that
\be\label{minorest 2}
\left|S_{\lm,\mu,q}(\A)\right|\leq 2C_1\eta_0.\qedhere
\ee
\end{proof}

\begin{proof}[Proof of Lemma \ref{lemma3}]

We first construct the number $\eta>0$. Choosing a lacunary sequence $\{\eta_j\}$ for which 
\be
\eta_1\ll \VE\quad \text{and}\quad \eta_{j+1} \leq(\VE/80C_1)\eta_{j}
\ee for each $j\geq1$ it is easy to see that
\[\sup_{\A\in\T^k}\sum_{j=1}^\infty\bigl|\widehat{\psi}_{q_{j+1},L_2}(\A)-\widehat{\psi}_{q_{j},L_2}(\A)\bigr|\leq C_2\]
where $q_j=q_{\eta_j}$. It follows immediately that there must exist $1\leq j\leq 40C_2/\VE$ such that
\be\label{80}
\|\widehat{\psi}_{q_{j+1},L_2}-\widehat{\psi}_{q_{j},L_2}\|_\infty\leq  \VE/40.
\ee
We set $\eta=\eta_j$ and $\eta'=\eta_{j+1}$ for this value of $j$ and note that $\eta$ satisfies the inequality 
\[\exp(-C\VE^{-1}\log\VE^{-1})\leq\eta\ll\VE.\]

Estimate (\ref{wholepoint}) now follows immediately from Lemma \ref{Weyl Estimates 2} and (\ref{80}), since 
\be
\|(1-\widehat{\psi}_{q,L_2})S_{\lm,\mu,q}\|_\infty\leq\|(1-\widehat{\psi}_{q',L_2})S_{\lm,\mu,q}\|_\infty+\|(\widehat{\psi}_{q,L_2}-\widehat{\psi}_{q',L_2})S_{\lm,\mu,q}\|_\infty\leq 2C_1\eta'/\eta+\VE/40
\ee
and $\eta'/\eta\leq \VE/80C_1$.

\comment{We now define the following new function.
Let
\[f'_2=f*\psi_{q',L_2}\]
where $q'=q_{\eta'}$.}

Lemma \ref{lemma3} now follows, since by
arguing as in the proof of Proposition \ref{smoothdichotomy} above, we obtain
\begin{align*}
\max\{|\Lambda_{q,\mu}(f_2,f_1)|,|\Lambda_{q,\mu}(f,f_2)|\}&\leq \int_{\T^k} |\widehat{f}(\A)|^2\,\bigl|1-\widehat{\psi}_{q,L_2}(\A)\bigr||S_{\lm,\mu,q}(\A)|\,d\A\\
& \leq \|(1-\widehat{\psi}_{q,L_2})S_{\lm,\mu,q}\|_\infty M^k\end{align*}
where the last inequality follows from Plancherel and the fact that $\|f\|_2^2\leq\|f\|_1\leq M^k$. 
\end{proof}


\section{The proofs of Theorems \ref{finite1} and \ref{infinite1}}

In both of the proofs below we fix a polynomial $P(n)$ with integer coefficients, namely
\[P(n)=c_1 n+\cdots+c_k n^k\]
and let $\mathcal{P}:\Z^k\rightarrow\Z$ denote the mapping given by
\[\mathcal{P}(b)=c_1 b_1+\cdots+c_k b_k.\]
Furthermore, given any set $A\subseteq\Z$ we define
\[A_j=\{a\in A\,:\,a\equiv j\mod m\}=A\cap(m\Z+j)\]
for each $0\leq j\leq m-1$ where \[m=\gcd(c_1,\dots,c_k).\] 

\subsection{Deduction of Thereom \ref{finite1} from Theorem \ref{finite2}}\label{finite}

Let $\VE>0$ and $A\subseteq[1,N]$ with $\D=|A|/N$ satisfying $0<\VE\leq\D^2$. 
We suppose that
\[|A\cap(A+P(n))|\leq (\D^2-\VE)N\]
for some $n\in\N$.
Without loss in generality we will make the convenient additional assumption that $m|N$. 

It is easy to see that there necessarily exists $0\leq j\leq m-1$ such that
\[|A_j\cap(A_j+P(t))|\leq(\D_j^2-\VE)N/m\]
with $\D_j=m|A_j|/N$.
If we now let
\[B'=\{b\in\Z^k\,:\,\mathcal{P}(b)\in A_{j}-j\}\cap Q\]
where
\[Q= \mathcal{P}^{-1}(m\Z\cap[1,N])\cap[-N',N']^k\]
and $N'$ is some suitably large multiple of $N$ (depending only on the coefficients of $P$) then it follows that 
\[\D_j=|B'|/|Q|\]
and
\[\frac{|B'\cap(B'+\gamma(n))|}{|Q|}=\frac{|A_j\cap(A_j+P(n))|}{N/m}.\]

We now set $M=\eta N/m$ for some suitably small $\eta>0$,
\[X=\{x\in(M\Z)^k\,:\,x+[1,M]^k\subseteq Q\}\]
and
\[Q'=\bigcup_{x\in X} (x+[1,M]^k),\]
noting that we can clearly choose $\eta\ll\VE$ to ensure that
\[\frac{|Q\setminus Q'|}{|Q|}\leq\frac{\VE}{10}\quad\Longleftrightarrow\quad\frac{|Q|}{|Q'|}\leq1+\frac{\VE}{9}.\]
Thus, if we set $B''=B'\cap Q'$ and $\B=|B''|/|Q'|$, it follows that
\[\B\geq\D_j-\VE/10\]
and
\[\frac{|B''\cap(B''+\gamma(t))|}{|Q'|}\leq \frac{|B'\cap(B'+\gamma(t))|}{|Q|}\frac{|Q|}{|Q'|}\leq\B^2-\VE/2.\]
It follows that there must exist $x\in X$ such that if we set
\[B=B''\cap(x+[1,M]^k)\]
then
\[\frac{|B\cap(B+\gamma(n))|}{M^k}\leq \left(\frac{|B|}{M^k}\right)^2-\VE/2.\]

In summary we have shown that for any given set $A\subseteq[1,N]$ and $\VE>0$ there exists a set $B\subseteq[1,M]^k$ with $M\ll \VE N/m$ such that
\[\left\{n\in\N\,:\, \frac{|B\cap(B+\gamma(n))|}{M^k}> \left(\dfrac{|B|}{M^k}\right)^2-\VE/2\right\}\subseteq\left\{n\in\N\,:\, \frac{|A\cap(A+P(n))|}{N}> \left(\dfrac{|A|}{N}\right)^2-\VE\right\}\]
and hence Theorem \ref{finite1} now follows 
from Theorem \ref{finite2}.
\qed

\subsection{Deduction of Theorem \ref{infinite1} from Theorem \ref{infinite2}}\label{infinite}
Let $\VE>0$ and $A\subseteq\N$ with $\D=\D(A)$ satisfying $0<\VE\leq\D^2$.
We suppose that
\[\D(A\cap(A+P(n)))\leq \D^2-\VE\]
for some $n\in\N$.
It follows from the definition of upper density that there exists a sequence of intervals $\{I_i\}$ with $|I_i|=N_i$, where $\{N_i\}\subseteq m\N$ and $N_i\nearrow\infty$, such that
\[\frac{|(A\cap I_i)\cap((A\cap I_i)+P(n))|}{N_i}\leq \D^2-\VE/2\]
while
\[\frac{|A\cap I_i|}{N_i}\geq \D-\VE/10.\]
If we define $\D_i=|A\cap I_i|/N_i$, it therefore follows that
\[\frac{|(A\cap I_i)\cap((A\cap I_i)+P(n))|}{N_i}\leq \D_i^2-\VE/5.\]
Note that $(A\cap I_i)_j=A_j\cap I_i$. If we set $\D_{ij}=m|A_j\cap I_i|/N_i$, then
\[\D_i=\frac{1}{m}\sum_{j=0}^{m-1}\D_{ij}\]
and as a consequence of the Cauchy-Schwarz inequality we have
\[\D_i^2\leq \frac{1}{m}\sum_{j=0}^{m-1}\D_{ij}^2.\]
It therefore follows immediately from the fact that
\[\sum_{j=0}^{m-1}|(A_j\cap I_i)\cap((A_j\cap I_i)+P(n))|=|(A\cap I_i)\cap((A\cap I_i)+P(n))|\leq \frac{1}{m}\sum_{j=0}^{m-1}(\D_{ij}^2-\VE/5)N_i\]
that for each $i$ there must exist $0\leq j\leq m-1$ for which 
\[|(A_j\cap I_i)\cap((A_j\cap I_i)+P(n))|\leq (\D_{ij}^2-\VE/5)N_i/m.\]

We will assume (by refining our collection $\{I_i\}$ if necessary) that the same $j$ is selected for each $i$. Since $(A\cap I_i)_j=A_j\cap I_i$ it follows, from the definition of upper density, that 
\[\D(A_j)\geq\D_{ij}/m\]
and
\[\D(A_j\cap(A_{j}+P(n)))\leq (\D_{ij}^2-\VE/5)/m.\]

If we now define
\[B=\{b\in\Z^k\,:\,\mathcal{P}(b)\in A_{j}-j\}\]
it follows immediately that
\[\D(B)=m\D(A_j)\]
\[\D(B\cap(B+\gamma(n)))=m\D(A_j\cap(A_j+P(n)))\]
 and consequently
\[\D(B\cap(B+\gamma(n)))\leq \D(B)^2-\VE/5.\]

In summary we have shown that for any given set $A\subseteq\N$ with $\D(A)>0$ and $\VE>0$ there exists a set $B\subseteq\Z^k$ with $\D(B)>0$ such that
\[\left\{n\in \N \,:\,\D(B\cap(B+\gamma(n)))>\D(B)^2-\VE/5\right\}\subseteq\left\{n\in \N \,:\,\D(A\cap(A+P(n)))>\D(A)^2-\VE\right\}\]
and hence Theorem \ref{infinite1} follows immediately from Theorem \ref{infinite2}.
\qed

\section{The parameter $L$ in Theorem \ref{infinite1} necessarily depends on the set $A$}\label{A}

In this final section we construct an example to show that the parameter $L$ in Theorem \ref{infinite1} necessarily depends on the actual set $A$ and not just on its density. 

\begin{propn} Let $P\in\Z[n]$ with $P(0)=0$ and $L\in\N$, then there exist  $A\subseteq\N$ with $\D(A)=1/3$ and an unbounded increasing sequence $\{\lm_j\}$ with the property that
$A\cap(A+P(n))=\emptyset$ whenever $n\in\bigcup_{j=0}^\infty[\lm_j,\lm_j+L]$.
\end{propn}

\begin{proof}
With out loss in generality we assume that the leading coefficient of $P$.
Set $M=P(aL)$ with $a\in\N$ chosen so that $P$ is increasing and $2P(aL)\geq P((a+1)L )$. We definine $A\subseteq\N$ such that $A=A+3M$ and $A\cap[1,3M]=[M+1,2M]$.

Since $P(n)=P(m)\pmod{3M}$ whenever $n=m\pmod{3M}$, it is easy to see that if $\lm_j=j3M+aL$, then the fact that $A\cap(A+P(n))=\emptyset$ whenever $n\in[\lm_j,\lm_j+L]$ for some $j$, follows from the fact that this holds for $j=0$ (as can be easily verified by the reader).
\end{proof}


\end{document}